\documentclass[reqno,12pt,a4paper]{amsart}
\usepackage{amscd,amsfonts,amssymb,amsthm,latexsym}
\usepackage{mathtools}
\usepackage{bm}
\usepackage{srcltx}

\theoremstyle{plain}
\newtheorem{theorem}{Theorem}[section]
\newtheorem*{theorem*}{Theorem}

\newtheorem{lemma}{Lemma}[section]

\theoremstyle{definition}

\newtheorem*{definition*}{Definition}

\theoremstyle{remark}

\newtheorem*{remark*}{Remark}

\numberwithin{equation}{section}

\begin{document}
\raggedbottom %Ќужно, чтобы текст вертикально на раст€гивалс€

\title[A large integer is a sum of two prime avoiding numbers]{A large integer is a sum of two prime avoiding numbers}

\author{Artyom Radomskii}

\begin{abstract} Let $f(n)=\min_{p} |n-p|$, where $p$ is a prime. We show that there is a positive constant $\delta$ such that for any large integer $N$ there exist two positive integers $n_1$ and $n_2$ such that $N=n_1 + n_2$ and $f(n_i)\geq \ln N (\ln\ln N)^{\delta}$, $i=1, 2$.
\end{abstract}

 \address{Steklov Mathematical Institute of Russian Academy of Sciences\\
 8 Gubkina St., Moscow 119991, Russia}


\email{artyom.radomskii@mail.ru}

\maketitle

\section{Update}
The same result was obtained by M.\,R. Gabdullin.  This version will not be published in this form. It will appear in a joint manuscript by M.\,R. Gabdullin and A.\,O. Radomskii.

\section{Introduction}

Let $f(n)= \min_{p}|n-p|$, where $p$ is a prime. We prove the following result.

\begin{theorem}\label{T1}
Let $\delta\in (0, 1/2)$ be such that
\begin{equation}\label{Basic_delta}
12\cdot 10^{2\delta}< \ln \bigg(\frac{1}{2\delta}\bigg).
\end{equation}Then for any large integer $N$ there exist two positive integers $n_1$ and $n_2$ such that $N=n_1+ n_2$
 and $f(n_i)\geq \ln N (\ln\ln N)^{\delta}$, $i=1, 2$.
\end{theorem}
Theorem \ref{T1} improves the trivial result that for any large integer $N$ there exist two positive integers $n_1$ and $n_2$ such that $N= n_1+ n_2$ and $f(n_i)\gg \ln N$, $i=1, 2$. To prove Theorem \ref{T1}, we used the technique from the paper \cite{FKMPT}.

\section{Notation}

We use $X\ll Y$, $Y\gg X$, or $X=O(Y)$ to denote the estimate $|X|\leq C Y$ for some constant $C>0$. Throughout the remainder of the paper, all implied constants in $O$, $\ll$ or $\gg$ may depend on quantities $\delta$, $M$, $K$, and $\xi$ which we specify in the next section. We also assume that the quantity $x$ is sufficiently large in terms of all of these parameters.

The notation $X=o(Y)$ as $x\to \infty$ means $\lim_{x\to \infty} X/Y = 0$ (holding other parameters fixed). The notation $X\sim Y$ as $x\to \infty$ means $\lim_{x\to \infty} X/Y = 1$ (holding other parameters fixed).

If $S$ is a statement, we use $1_{S}$ to denote its indicator, thus $1_{S}=1$ when $S$ is true and $1_{S}=0$ when $S$ is false.

We will rely on probabilistic methods in this paper. Boldface symbols such as $\mathbf{b}$, $\mathbf{n}$, $\mathbf{S}$, $\bm{\lambda}$, etc. denote random variables (which may be real numbers, random sets, random functions, etc.) Most of these random variables will be discrete (in fact they will only take on finitely many values), so that we may ignore any technical issues of measurability. We use $\mathbb{P}(E)$ to denote the probability of a random event $E$, and $\mathbb{E}(\mathbf{X})$ to denote the expectation of the random (real-valued) variable $\mathbf{X}$.

The symbols $p$ and $q$ (as well as variants such as $p_1$, $p_2$, etc.) will always denote primes.

If $x$ is a real number, then $[x]$ denotes its integral part, and $\lceil x \rceil$ is the smallest integer $n$ such that $n \geq x$.

By $\# A$ we denote the number of elements of a finite set $A$.

\section{Proof of Theorem \ref{T1}}

We may assume that
\begin{equation}\label{Basic_Delta_2}
\ln \bigg(\frac{1}{2\delta}\bigg) \leq 1000.
\end{equation} We put
\begin{equation}\label{Def_y_z}
y=\lceil x (\ln x)^{\delta}\rceil,\quad\qquad z=\frac{y (\ln \ln x)^{2}}{(\ln x)^{1/2}}.
\end{equation} We define
\begin{gather*}
P_{x}:= \prod_{p\leq x}p,\quad\qquad \sigma_{x}:= \prod_{p\leq x}\bigg(1-\frac{1}{p}\bigg),\\
 P_{z,x}:=\prod_{z<p\leq x}p,\quad\qquad \sigma_{z,x}:= \prod_{z<p\leq x}\bigg(1-\frac{1}{p}\bigg).
\end{gather*} It is well-known that
\begin{equation}\label{sigma_assympt}
\sigma_{x}\sim \frac{C_{1}}{\ln x},
\end{equation}where $C_{1}$ is a positive absolute constant (namely, $C_{1}=e^{-\gamma}$, where $\gamma= 0.577215\ldots$ is Euler's constant).

 For any integer $b$, we define
\begin{align*}
S_{x}(b)&:= \{n\in \mathbb{Z}:\ n-b \not \equiv 0\text{ (mod $p$) for any $p\leq x$}\},\\
S_{z,x}(b)&:= \{n\in \mathbb{Z}:\ n-b \not \equiv 0\text{ (mod $p$) for any $z<p\leq x$}\}.
\end{align*} We denote $S_x:= S_{x}(0)$. It is clear that $S_{x}(b)= S_{x}+ b$. Here $S_{x}+b:= \{s+b:\ s\in S_{x}\}$.

Fix a real number $\xi>1$ (which we will eventually take very close to $1$) and define
\[
\mathfrak{H}:=\bigg\{H\in\{1, \xi, \xi^{2},\ldots\}:\ \frac{2y}{x}\leq H\leq \frac{y}{\xi z}\bigg\}.
\]It is clear that
\[
\bigsqcup_{H\in\mathfrak{H}}\Big(\frac{y}{\xi H}, \frac{y}{H}\Big]\subset \Big(z, \frac{x}{2}\Big].
\]By \eqref{Def_y_z} for $H\in \mathfrak{H}$ we have
\begin{equation}\label{H_range}
2(\ln x)^{\delta}\leq H\leq \frac{(\ln x)^{1/2}}{(\ln \ln x)^{2}}.
\end{equation}For each $h\in \mathfrak{H}$, let $\mathcal{Q}_{H}$ be the set of primes $q\in (y/(\xi H), y/H]$. By the Prime Number Theorem we have
\begin{equation}\label{QH_assympt}
\# \mathcal{Q}_{H}\sim (1- 1/\xi)\frac{y}{H\ln x}.
\end{equation}Let
\[
\mathcal{Q}=\bigcup_{H\in \mathfrak{H}}\mathcal{Q}_{H}.
\]For $q\in \mathcal{Q}$, let $H_{q}$ be the unique element of $\mathfrak{H}$ such that
\[
\frac{y}{\xi H_{q}}< q\leq \frac{y}{H_{q}}.
\]We define
\begin{align*}
\mathfrak{H}_{1}=\bigg\{H:\ H=\xi^{j},\ j\geq 0,\ \text{$j$ is even}, \frac{2y}{x}\leq H\leq \frac{y}{\xi z}\bigg\},\\
\mathfrak{H}_{2}=\bigg\{H:\ H=\xi^{j},\ j\geq 0,\ \text{$j$ is odd}, \frac{2y}{x}\leq H\leq \frac{y}{\xi z}\bigg\}.
\end{align*}Then
\[
\mathfrak{H}_{1}\sqcup \mathfrak{H}_{2}=\mathfrak{H}.
\]We put
\[
\mathcal{U}=\bigcup_{H\in \mathfrak{H}_{1}}\mathcal{Q}_{H},\quad\qquad \mathcal{V}=\bigcup_{H\in \mathfrak{H}_{2}}\mathcal{Q}_{H}.
\]We have
\[
\mathcal{U}\sqcup \mathcal{V} = \mathcal{Q}.
\]

Fix a real number $M$ satisfying
\[
M>6.
\]We denote by $\mathbf{b}$ a random residue class from $\mathbb{Z}/P\mathbb{Z}$, chosen with uniform probability, where we adopt the abbreviations
\[
P=P_{z},\quad \sigma=\sigma_{z},\quad \mathbf{S}= S_{z}+\mathbf{b}
\]as well as the projections
\[
P_{1}=P_{H^{M}},\quad \sigma_{1}=\sigma_{H^{M}},\quad \mathbf{b_{1}}\equiv \mathbf{b}\ \text{(mod $P_{1}$)},\quad
\mathbf{S_{1}}= S_{H^{M}}+ \mathbf{b_{1}}
\]and
\[
P_{2}=P_{H^{M}, z},\quad \sigma_{2}=\sigma_{H^{M}, z},\quad \mathbf{b_{2}}\equiv \mathbf{b}\ \text{(mod $P_{2}$)},\quad
\mathbf{S_{2}}= S_{H^{M}, z}+ \mathbf{b_{2}}
\]with the convention that $\mathbf{b_1}\in \mathbb{Z}/P_{1}\mathbb{Z}$ and $\mathbf{b_2}\in \mathbb{Z}/P_{2}\mathbb{Z}$. Thus, $\mathbf{b_1}$ and $\mathbf{b_{2}}$ are each uniformly distributed, are independent of each other, and likewise $\mathbf{S_1}$ and $\mathbf{S_{2}}$ are independent. We also have the obvious relations
\[
P=P_1 P_2,\quad \sigma=\sigma_1\sigma_2,\quad \mathbf{S}=\mathbf{S_1}\cap \mathbf{S_2}.
\]

For prime $q$ and $n\in \mathbb{Z}$, define the random set
\[
\mathbf{AP}(J; q, n):=\{n+qh:\ 1\leq h \leq J\}\cap \mathbf{S_{1}}.
\]Let $K\geq 2$ be a fixed integer parameter, which we will eventually take to be very large. We put
\begin{equation}\label{Def_lambda}
\bm{\lambda}(H; q, n):=\begin{cases}
\sigma_{2}^{-\# \mathbf{AP}(KH; q, n)}    &\text{if $\mathbf{AP}(KH; q, n)\subset \mathbf{S_{2}}$;}\\
0  &\text{otherwise}.
                       \end{cases}
\end{equation}In other words,
\[
\bm{\lambda}(H; q, n)=\frac{1_{\mathbf{AP}(KH; q, n)\subset \mathbf{S_{2}}}}{\sigma_{2}^{\# \mathbf{AP}(KH; q, n)}}.
\] We will focus attention on those $n$ satisfying
\[
-Ky< n\leq y,
\]for outside this interval, if $q\in \mathcal{Q}_{H}$ then $\mathbf{AP}(KH; q, n)$ does not intersect the interval $[1, y]$ of primary interest.

Let $N$ be a positive integer. We put
\[
\mathbf{S^{*}}= S_{z}+ (-N-\mathbf{b}),\quad \mathbf{S_{1}^{*}}=S_{H^{M}}+ (-N-\mathbf{b_{1}}),\quad
\mathbf{S_{2}^{*}}=S_{H^{M}, z}+ (-N-\mathbf{b_{2}}).
\]It is clear that
\[
\mathbf{S^{*}}=\mathbf{S_{1}^{*}}\cap \mathbf{S_{2}^{*}}.
\]We define
\[
\mathbf{{AP}^{*}}(J; q, n):=\{n-qh:\ 1\leq h \leq J\}\cap \mathbf{S_{1}^{*}}
\]and
\begin{equation}\label{Def_lambda_star}
\bm{{\lambda}^{*}}(H; q, n):=\begin{cases}
\sigma_{2}^{-\# \mathbf{{AP}^{*}}(KH; q, n)}    &\text{if $\mathbf{{AP}^{*}}(KH; q, n)\subset \mathbf{S_{2}^{*}}$;}\\
0  &\text{otherwise}.
                       \end{cases}
\end{equation}In other words,
\[
\bm{{\lambda}^{*}}(H; q, n)=\frac{1_{\mathbf{{AP}^{*}}(KH; q, n)\subset \mathbf{S_{2}^{*}}}}{\sigma_{2}^{\# \mathbf{{AP}^{*}}(KH; q, n)}}.
\]We will focus attention on those $n$ satisfying
\[
-y\leq n< Ky,
\]for outside this interval, if $q\in \mathcal{Q}_{H}$ then $\mathbf{{AP}^{*}}(KH; q, n)$ does not intersect the interval $[-y, -1]$ of primary interest.

\begin{lemma}\label{L1}
Assume that $M\geq 2$. Then

\textup{(i)} One has
\begin{align}
\mathbb{E} \#(\mathbf{S}\cap [1,y])&=\sigma y,\label{L1:E_S}\\
\mathbb{E} \big(\#(\mathbf{S}\cap [1,y])\big)^{2}&= \bigg(1+ O\bigg(\frac{1}{\ln y}\bigg)\bigg)(\sigma y)^{2},\label{L1:D_S}\\
\mathbb{E} \#(\mathbf{S^{*}}\cap [-y,-1])&=\sigma y,\label{L1:E_S_star}\\
\mathbb{E} \big(\#(\mathbf{S^{*}}\cap [-y,-1])\big)^{2}&= \bigg(1+ O\bigg(\frac{1}{\ln y}\bigg)\bigg)(\sigma y)^{2}.\label{L1:D_S_star}
\end{align}

\textup{(ii)} For every $H\in \mathfrak{H}_{1}$, and for $j\in \{1, 2\}$ we have
\begin{equation}\label{L1:lambda}
\mathbb{E}\sum_{q\in \mathcal{Q}_{H}} \left(\sum_{-Ky< n\leq y}\bm{\lambda}(H; q, n)\right)^{j}= \left(1 +
 O\left(\frac{1}{H^{M-2}}\right)\right) \big((K+1)y\big)^{j}\# \mathcal{Q}_{H}.
\end{equation}

For every $H\in \mathfrak{H}_{2}$, and for $j\in \{1, 2\}$ we have
\begin{equation}\label{L1:lambda_star}
\mathbb{E}\sum_{q\in \mathcal{Q}_{H}} \left(\sum_{-y\leq n< Ky}\bm{{\lambda}^{*}}(H; q, n)\right)^{j}= \left(1 +
 O\left(\frac{1}{H^{M-2}}\right)\right) \big((K+1)y\big)^{j}\# \mathcal{Q}_{H}.
\end{equation}

\textup{(iii)} For every $H\in \mathfrak{H}_{1}$, and for $j\in \{1,2\}$ we have
\begin{equation}\label{L1:lamda_AP}
\mathbb{E}\sum_{n\in \mathbf{S}\cap [1, y]} \left(\sum_{q\in \mathcal{Q}_{H}} \sum_{h\leq K H}\bm{\lambda}(H; q, n-qh)\right)^{j} = \left(1 + O\left(\frac{1}{H^{M-2}}\right)\right) \left(\frac{\#\mathcal{Q}_{H}KH}{\sigma_{2}}\right)^{j}\sigma y.
\end{equation}

For every $H\in \mathfrak{H}_{2}$, and for $j\in \{1,2\}$ we have
\begin{equation}\label{L1:lambda_AP_star}
\mathbb{E}\sum_{n\in \mathbf{S^{*}}\cap [-y, -1]} \left(\sum_{q\in \mathcal{Q}_{H}} \sum_{h\leq K H}\bm{{\lambda}^{*}}(H; q, n+qh)\right)^{j} = \left(1 + O\left(\frac{1}{H^{M-2}}\right)\right) \left(\frac{\#\mathcal{Q}_{H}KH}{\sigma_{2}}\right)^{j}\sigma y.
\end{equation}
\end{lemma}
\begin{proof} The statements \eqref{L1:E_S}, \eqref{L1:D_S}, \eqref{L1:lambda}, and \eqref{L1:lamda_AP} follow from \cite[Theorem 3]{FKMPT}. The statements \eqref{L1:E_S_star}, \eqref{L1:D_S_star}, \eqref{L1:lambda_star}, and \eqref{L1:lambda_AP_star} can be proved similarly.
\end{proof}

\begin{lemma}\label{L2}
Let $\delta$ be a real number satisfying $0<\delta< 1/2$, \eqref{Basic_delta} and \eqref{Basic_Delta_2}. Then there exist $M>6$, $K>1$, $\xi>1$, $\varepsilon>0$ such that for any large $x$ (with respect to $M$, $\varepsilon$, $K$, $\xi$) there exist an integer $b$ and sets $\mathcal{U'}\subset \mathcal{U}$ and $\mathcal{V'}\subset \mathcal{V}$ such that the following statements hold.

\textup{(i)} One has
\begin{align}
\# (S\cap [1, y])&\leq 2\sigma y,\label{L2:i_S}\\
\# (S^{*}\cap [-y, -1])&\leq 2\sigma y.\label{L2:i_S_star}
\end{align}

\textup{(ii)} For all $q\in \mathcal{U'}$, one has
\begin{equation}\label{L2:lambda}
\sum_{-Ky< n \leq y} \lambda (H_{q}; q, n) = \left(1+O\left(\frac{1}{(\ln x)^{\delta (1+\varepsilon)}}\right)\right)(K+1)y.
\end{equation}

\textup{(iii)} For all $q\in \mathcal{V'}$, one has
\begin{equation}\label{L2:lambda_star}
\sum_{-y\leq n < Ky} {\lambda}^{*} (H_{q}; q, n) = \left(1+O\left(\frac{1}{(\ln x)^{\delta (1+\varepsilon)}}\right)\right)(K+1)y.
\end{equation}

\textup{(iv)} There is a set $\mathcal{N}\subset (S\cap [1, y])$ such that
\[
\# \big((S\cap [1, y]) \setminus \mathcal{N}\big) \leq \frac{x}{10\ln x}
\]and for any $n\in \mathcal{N}$, one has
\begin{equation}\label{L2:lambda_AP}
\sum_{q\in \mathcal{U'}} \sum_{h\leq KH_{q}}\lambda (H_{q}; q, n-qh)=\left(C_2 + O\left(\frac{1}{(\ln x)^{\delta (1+\varepsilon)}}\right)\right) (K+1)y
\end{equation}for some quantity $C_2$ independent of $n$ with
\[
10^{2\delta}\leq C_{2}\leq 100.
\]

\textup{(v)} There is a set $\mathcal{N^{*}}\subset (S^{*}\cap [-y, -1])$ such that
\[
\# \big((S^{*}\cap [-y, -1]) \setminus \mathcal{N^{*}}\big) \leq \frac{x}{10\ln x}
\]and for any $n\in \mathcal{N^{*}}$, one has
\begin{equation}\label{L2:N_star_AP}
\sum_{q\in \mathcal{V'}} \sum_{h\leq KH_{q}}{\lambda}^{*} (H_{q}; q, n+qh)=\left(C_3 + O\left(\frac{1}{(\ln x)^{\delta (1+\varepsilon)}}\right)\right) (K+1)y
\end{equation}for some quantity $C_3$ independent of $n$ with
\[
10^{2\delta}\leq C_{3}\leq 100.
\]
\end{lemma}

\begin{proof}[Proof of Lemma \ref{L2}.] We draw $\mathbf{b}$ uniformly at random from $\mathbb{Z}/P\mathbb{Z}$. It will suffice to generate random sets $\bm{\mathcal{U'}}$ and $\bm{\mathcal{V'}}$ such that the random functions $\bm{\lambda}$ and $\bm{{\lambda}^{*}}$ defined in \eqref{Def_lambda} and \eqref{Def_lambda_star} satisfy the conclusions of Lemma \ref{L2} (with $b$ replaced by $\mathbf{b}$) hold with positive probability - in fact, we will show that they hold with probability $1-o(1)$.

From Lemma \ref{L1}\textup{(i)} we have
\begin{align*}
\mathbb{E}\left(\#(\mathbf{S}\cap [1, y]) - \sigma y\right)^{2}&\ll \frac{(\sigma y)^{2}}{\ln y},\\
 \mathbb{E}\left(\#(\mathbf{S}^{*}\cap [-y, -1]) - \sigma y\right)^{2}&\ll \frac{(\sigma y)^{2}}{\ln y}.
\end{align*}Hence by Chebyshev's inequality, we see that
\begin{align}
\mathbb{P}\left( \#(\mathbf{S}\cap [1, y])\leq 2\sigma y\right)= 1 - O\left(\frac{1}{\ln x}\right),\label{L2:S_est}\\
\mathbb{P}\left( \#(\mathbf{S}^{*}\cap [-y, -1])\leq 2\sigma y\right)= 1 - O\left(\frac{1}{\ln x}\right),\label{L2:S_star_est}
\end{align}verifying \eqref{L2:i_S} and \eqref{L2:i_S_star} in Lemma \ref{L2}.

Let $H\in \mathfrak{H}_{1}$. From \eqref{L1:lambda} we have
\begin{align}\label{L2:Basic_Moment_H1}
\mathbb{E}\sum_{q\in \mathcal{Q}_{H}}\left(\sum_{-Ky< n\leq y}\bm{\lambda}(H;q,n)- (K+1)y\right)^{2}&\ll \frac{y^{2}\#\mathcal{Q}_{H}}{H^{M-2}}.
\end{align}We put
\[
\bm{\mathcal{U'}}_{H}=\left\{q\in \mathcal{Q}_{H}: \Big|\sum_{-Ky<n\leq y}\bm{\lambda}(H;q,n)-(K+1)y\Big|\leq \frac{y}{H^{1+\varepsilon}}\right\}.
\]We have
\begin{align}
\mathbb{E}\# (\mathcal{Q}_{H}\setminus \bm{\mathcal{U'}}_{H})&=\mathbb{E}\sum_{q\in \mathcal{Q}_{H}\setminus \bm{\mathcal{U'}}_{H}} 1=\notag\\
&=\mathbb{E}\sum_{q\in \mathcal{Q}_{H}\setminus \bm{\mathcal{U'}}_{H}} \frac{\Big(\sum_{-Ky<n\leq y}\bm{\lambda}(H;q,n)-(K+1)y\Big)^{2}}{\Big(\sum_{-Ky<n\leq y}\bm{\lambda}(H;q,n)-(K+1)y\Big)^{2}}\leq\notag\\
&\leq \frac{H^{2+2\varepsilon}}{y^{2}}\mathbb{E}\sum_{q\in \mathcal{Q}_{H}\setminus \bm{\mathcal{U'}}_{H}}\Big(\sum_{-Ky<n\leq y}\bm{\lambda}(H;q,n)-(K+1)y\Big)^{2}\leq\notag\\
&\leq \frac{H^{2+2\varepsilon}}{y^{2}}\mathbb{E}\sum_{q\in \mathcal{Q}_{H}}\Big(\sum_{-Ky<n\leq y}\bm{\lambda}(H;q,n)-(K+1)y\Big)^{2}\ll \frac{\#\mathcal{Q}_{H}}{H^{M-4-2\varepsilon}}.\label{L2:EXPECTATION_1}
\end{align}By Markov's inequality, we have
\[
\mathbb{P}\bigg(\# (\mathcal{Q}_{H}\setminus \bm{\mathcal{U'}}_{H})\leq \frac{\# \mathcal{Q}_{H}}{H^{M-4-3\varepsilon}}\bigg)=1- O(H^{-\varepsilon}).
\]We observe that for any $\alpha>0$ we have
\begin{equation}\label{L2:H_alpha}
\sum_{H\in \mathfrak{H}} H^{-\alpha}\leq \frac{1}{\big(2(\ln x)^{\delta}\big)^{\alpha}}+
\frac{1}{\big(2(\ln x)^{\delta}\xi\big)^{\alpha}}+\frac{1}{\big(2(\ln x)^{\delta}\xi^{2}\big)^{\alpha}}+\ldots\ll_{\alpha} (\ln x)^{-\delta \alpha}.
\end{equation}Hence, with probability $1-O((\ln x)^{-\delta\varepsilon})$ the relation
\begin{equation}\label{L2:U_strich_est}
\# (\mathcal{Q}_{H}\setminus \bm{\mathcal{U'}}_{H})\leq \frac{\# \mathcal{Q}_{H}}{H^{M-4-3\varepsilon}}
\end{equation}holds for every $H\in\mathfrak{H}_{1}$ simultaneously. We put
\[
\bm{\mathcal{U'}}=\bigcup_{H\in \mathfrak{H}_{1}} \bm{\mathcal{U'}}_{H}.
\]Since $H\geq 2(\ln x)^{\delta}$ for any $H\in \mathfrak{H}_{1}$, we have for any $q\in \bm{\mathcal{U'}}$
\begin{equation}\label{L2:U_st_prop}
\sum_{-Ky< n \leq y} \bm{\lambda} (H_{q}; q, n) = \left(1+O\left(\frac{1}{(\ln x)^{\delta (1+\varepsilon)}}\right)\right)(K+1)y.
\end{equation}

Let $H\in \mathfrak{H}_{2}$. From \eqref{L1:lambda_star} we have
\begin{align*}
\mathbb{E}\sum_{q\in \mathcal{Q}_{H}}\left(\sum_{-y\leq n< Ky}\bm{{\lambda}^{*}}(H;q,n)- (K+1)y\right)^{2}&\ll \frac{y^{2}\#\mathcal{Q}_{H}}{H^{M-2}}.
\end{align*}We put
\[
\bm{\mathcal{V'}}_{H}=\left\{q\in \mathcal{Q}_{H}: \Big|\sum_{-y\leq n< Ky}\bm{{\lambda}^{*}}(H;q,n)-(K+1)y\Big|\leq \frac{y}{H^{1+\varepsilon}}\right\}.
\]Similarly, we have
\[
\mathbb{E}\# (\mathcal{Q}_{H}\setminus \bm{\mathcal{V'}}_{H})\ll \frac{\#\mathcal{Q}_{H}}{H^{M-4-2\varepsilon}},
\]and with probability $1-O((\ln x)^{-\delta\varepsilon})$ the relation
\begin{equation}\label{L2:V_strich_est}
\# (\mathcal{Q}_{H}\setminus \bm{\mathcal{V'}}_{H})\leq \frac{\# \mathcal{Q}_{H}}{H^{M-4-3\varepsilon}}
\end{equation}holds for every $H\in\mathfrak{H}_{2}$ simultaneously. We put
\[
\bm{\mathcal{V'}}=\bigcup_{H\in \mathfrak{H}_{2}} \bm{\mathcal{V'}}_{H}.
\]For any $q\in \bm{\mathcal{V'}}$ we have
\begin{equation}\label{L2:V_st_prop}
\sum_{-y\leq n < Ky} \bm{{\lambda}^{*}} (H_{q}; q, n) = \left(1+O\left(\frac{1}{(\ln x)^{\delta (1+\varepsilon)}}\right)\right)(K+1)y.
\end{equation}Thus, on the probability $1-o(1)$ event that \eqref{L2:U_strich_est} holds for every $H\in \mathfrak{H}_{1}$, that \eqref{L2:V_strich_est} holds for every $H\in \mathfrak{H}_{2}$ and that \eqref{L2:S_est} and \eqref{L2:S_star_est} hold, items \textup{(i)} \eqref{L2:i_S}, \eqref{L2:i_S_star}, \textup{(ii)} \eqref{L2:lambda} and \textup{(iii)} \eqref{L2:lambda_star} of Lemma \ref{L2} follow upon recalling \eqref{L2:U_st_prop} and \eqref{L2:V_st_prop}.

We work on part \textup{(iv)} of Lemma \ref{L2} using Lemma \ref{L1}\textup{(iii)} in a similar fashion to previous arguments. Let $H\in \mathfrak{H}_{1}$. From \eqref{L1:lamda_AP} we have
\[
\mathbb{E}\sum_{n\in \mathbf{S}\cap [1,y]} \left( \sum_{q\in \mathcal{Q}_{H}}\sum_{h\leq KH}\bm{\lambda}(H;q, n-qh) - \frac{\#\mathcal{Q}_{H} KH}{\sigma_{2}}\right)^{2}\ll \frac{1}{H^{M-2}} \left(\frac{\#\mathcal{Q}_{H}KH}{\sigma_2}\right)^{2}\sigma y.
\]We put
\begin{equation}\label{L2:def_E}
\bm{\mathcal{E}}_{H}=\left\{n\in \mathbf{S}\cap [1,y]: \left|\sum_{q\in \mathcal{Q}_{H}}\sum_{h\leq KH}\bm{\lambda}(H;q, n-qh) - \frac{\#\mathcal{Q}_{H} KH}{\sigma_{2}}\right|\geq \frac{\#\mathcal{Q}_{H} KH}{\sigma_{2} H^{(M-3)/2 - \varepsilon}}\right\}.
\end{equation}We have
\begin{align*}
\mathbb{E} \#\bm{\mathcal{E}}_{H}&=\mathbb{E}\sum_{n\in \bm{\mathcal{E}}_{H}} 1=\\
&=
\mathbb{E}\sum_{n\in \bm{\mathcal{E}}_{H}}\frac{\left(\sum_{q\in \mathcal{Q}_{H}}\sum_{h\leq KH}\bm{\lambda}(H;q, n-qh) - (\#\mathcal{Q}_{H} KH)/\sigma_{2}\right)^{2}}{\left(\sum_{q\in \mathcal{Q}_{H}}\sum_{h\leq KH}\bm{\lambda}(H;q, n-qh) - (\#\mathcal{Q}_{H} KH)/\sigma_{2}\right)^{2}}\leq\\
&\leq \frac{\sigma_{2}^{2}H^{M-3 -2\varepsilon}}{(\#\mathcal{Q}_{H}KH)^{2}}
\mathbb{E}\sum_{n\in \bm{\mathcal{E}}_{H}}\left(\sum_{q\in \mathcal{Q}_{H}}\sum_{h\leq KH}\bm{\lambda}(H;q, n-qh) - \frac{\#\mathcal{Q}_{H} KH}{\sigma_{2}}\right)^{2}\leq\\
&\leq \frac{\sigma_{2}^{2}H^{M-3 -2\varepsilon}}{(\#\mathcal{Q}_{H}KH)^{2}}
\mathbb{E}\sum_{n\in \mathbf{S}\cap [1,y]}\left(\sum_{q\in \mathcal{Q}_{H}}\sum_{h\leq KH}\bm{\lambda}(H;q, n-qh) - \frac{\#\mathcal{Q}_{H} KH}{\sigma_{2}}\right)^{2}\ll \frac{\sigma y}{H^{1+2\varepsilon}}.
\end{align*}By Markov's inequality, we have
\begin{equation}\label{L2:E_H}
\mathbb{P}\Big(\# \bm{\mathcal{E}}_{H}\leq \frac{\sigma y}{H^{1+\varepsilon}}\Big)= 1-O\left(\frac{1}{H^{\varepsilon}}\right).
\end{equation}

We next estimate the contribution from ``bad'' primes $q\in \mathcal{Q}_{H}\setminus \bm{\mathcal{U'}}_{H}$. For any $h\leq KH$, by the Cauchy-Schwarz inequality we have
\begin{align}
\mathbb{E}\sum_{n\in \mathbf{S}\cap [1,y]} \sum_{q\in \mathcal{Q}_{H}\setminus \bm{\mathcal{U'}}_{H}} \bm{\lambda}(H; q, n-qh)\leq &\big(\mathbb{E}\# (\mathcal{Q}_{H}\setminus \bm{\mathcal{U'}}_{H})\big)^{1/2}\cdot\notag\\
&\cdot\left(\mathbb{E} \sum_{q\in \mathcal{Q}_{H}\setminus \bm{\mathcal{U'}}_{H}} \bigg(\sum_{n=1}^{y} \bm{\lambda}(H; q, n-qh)\bigg)^{2}\right)^{1/2}.\label{L2:SOME_EXP_1}
\end{align}Given $q\in \mathcal{Q}_{H}\setminus \bm{\mathcal{U'}}_{H}$, we have
\[
\bigg|\sum_{n=1}^{y} \bm{\lambda}(H; q, n-qh)\bigg|\leq \bigg|\sum_{n=1}^{y} \bm{\lambda}(H; q, n-qh)- (K+1)y\bigg| + (K+1)y.
\] It is clear that
\begin{align*}
\bigg|\sum_{n=1}^{y} \bm{\lambda}(H; q, n&-qh)- (K+1)y\bigg|\leq\\
 &\leq\max\left((K+1)y, \bigg|\sum_{-Ky<n\leq y} \bm{\lambda}(H; q, n)- (K+1)y\bigg|\right)\leq\\
 &\leq\bigg|\sum_{-Ky<n\leq y} \bm{\lambda}(H; q, n)- (K+1)y\bigg|+ (K+1)y.
\end{align*}

Since $(a+b)^{2}\leq 2 (a^{2}+b^{2})$, we have
\[
\bigg(\sum_{n=1}^{y} \bm{\lambda}(H; q, n-qh)\bigg)^{2}\leq 2\left(\bigg(\sum_{-Ky<n\leq y} \bm{\lambda}(H; q, n)- (K+1)y\bigg)^{2}+ 4(K+1)^{2}y^{2}\right).
\]Applying \eqref{L2:Basic_Moment_H1} and \eqref{L2:EXPECTATION_1}, we obtain
\begin{align*}
\mathbb{E} &\sum_{q\in \mathcal{Q}_{H}\setminus \bm{\mathcal{U'}}_{H}} \bigg(\sum_{n=1}^{y} \bm{\lambda}(H; q, n-qh)\bigg)^{2}\leq\\
 &\leq2\mathbb{E} \sum_{q\in \mathcal{Q}_{H}\setminus \bm{\mathcal{U'}}_{H}} \bigg(\sum_{-Ky<n\leq y} \bm{\lambda}(H; q, n)- (K+1)y\bigg)^{2} +\\
 &+ 8\big((K+1)y\big)^{2}\mathbb{E}\# (\mathcal{Q}_{H}\setminus \bm{\mathcal{U'}}_{H})
 \ll \frac{y^{2}\#\mathcal{Q}_{H}}{H^{M-2}} + \frac{y^{2}\#\mathcal{Q}_{H}}{H^{M-4-2\varepsilon}}\ll
 \frac{y^{2}\#\mathcal{Q}_{H}}{H^{M-4-2\varepsilon}}.
\end{align*} We see from \eqref{L2:SOME_EXP_1} that
\[
\mathbb{E}\sum_{n\in \mathbf{S}\cap [1,y]} \sum_{q\in \mathcal{Q}_{H}\setminus \bm{\mathcal{U'}}_{H}} \bm{\lambda}(H; q, n-qh)\ll \frac{y \#\mathcal{Q}_{H}}{H^{M-4 - 2\varepsilon}}.
\]By summing over $h\leq KH$, we obtain
\[
\mathbb{E}\sum_{n\in \mathbf{S}\cap [1,y]} \sum_{q\in \mathcal{Q}_{H}\setminus \bm{\mathcal{U'}}_{H}}
 \sum_{h\leq KH}\bm{\lambda}(H; q, n-qh)\ll \frac{y \#\mathcal{Q}_{H}}{H^{M-5 - 2\varepsilon}}.
\]We put
\begin{equation}\label{L2:def_E_str}
\bm{\mathcal{E'}}_{H}=\bigg\{n\in \mathbf{S}\cap [1, y]:
\sum_{q\in \mathcal{Q}_{H}\setminus \bm{\mathcal{U'}}_{H}}
 \sum_{h\leq KH}\bm{\lambda}(H; q, n-qh)\geq \frac{ \#\mathcal{Q}_{H}KH}{H^{1+\varepsilon}\sigma_2}\bigg\}.
\end{equation}We have (see also \eqref{sigma_assympt})
\begin{align*}
\mathbb{E} \#\bm{\mathcal{E'}}_{H}&=\mathbb{E}\sum_{n\in \bm{\mathcal{E'}}_{H}}1=
\mathbb{E}\sum_{n\in \bm{\mathcal{E'}}_{H}} \frac{\sum_{q\in \mathcal{Q}_{H}\setminus \bm{\mathcal{U'}}_{H}}
 \sum_{h\leq KH}\bm{\lambda}(H; q, n-qh)}{\sum_{q\in \mathcal{Q}_{H}\setminus \bm{\mathcal{U'}}_{H}}
 \sum_{h\leq KH}\bm{\lambda}(H; q, n-qh)}\leq\\
 &\leq \frac{H^{1+\varepsilon}\sigma_2}{\#\mathcal{Q}_{H} KH}
 \mathbb{E}\sum_{n\in \mathbf{S}\cap [1,y]}\sum_{q\in \mathcal{Q}_{H}\setminus \bm{\mathcal{U'}}_{H}}
 \sum_{h\leq KH}\bm{\lambda}(H; q, n-qh)\ll \frac{y\sigma_2}{H^{M-5-3\varepsilon}}\ll\\
 & \ll \sigma y \frac{\ln H}{H^{M-5 - 3\varepsilon}}\ll \frac{\sigma y}{H^{M-5-4\varepsilon}}.
\end{align*}By Markov's inequality, we have
\begin{equation}\label{L2:E_H_str}
\mathbb{P}\bigg(\#\bm{\mathcal{E'}}_{H} \leq \frac{\sigma y}{H^{1+\varepsilon}}\bigg)=1-O\left(\frac{1}{H^{M-6-5\varepsilon}}\right).
\end{equation} We see from \eqref{L2:H_alpha}, \eqref{L2:E_H} and \eqref{L2:E_H_str} that with probability $1 - O\big((\ln x)^{-\delta \eta}\big)$, where $\eta = \min (\varepsilon, M-6-5\varepsilon) >0$, the relations
\[
\#\bm{\mathcal{E}}_{H} \leq \frac{\sigma y}{H^{1+\varepsilon}},\qquad
\#\bm{\mathcal{E'}}_{H} \leq \frac{\sigma y}{H^{1+\varepsilon}}
\]hold for every $H\in \mathfrak{H}_{1}$ simultaneously. We put
\[
\bm{\mathcal{N}} = (\mathbf{S}\cap [1,y]) \setminus \bigcup_{H\in \mathfrak{H}_{1}}
(\bm{\mathcal{E}}_{H}\cup \bm{\mathcal{E'}}_{H}).
\] Applying \eqref{L2:H_alpha} (see also \eqref{Def_y_z}), we obtain
\[
\# \bigcup_{H\in \mathfrak{H}_{1}}
(\bm{\mathcal{E}}_{H}\cup \bm{\mathcal{E'}}_{H})\ll \frac{\sigma y}{(\ln x)^{\delta (1+\varepsilon)}}
\]which is smaller than $x/ (10 \ln x)$ for large $x$. It remains to verify \eqref{L2:lambda_AP} for $n\in \bm{\mathcal{N}}$. Since $n\notin \bm{\mathcal{E}}_{H}$ and $n\notin \bm{\mathcal{E'}}_{H}$ for every $H\in \mathfrak{H}_{1}$, the inequalities opposite to those in \eqref{L2:def_E} and \eqref{L2:def_E_str} hold, and we have for each $H\in \mathfrak{H}_1$ the asymptotic (see also \eqref{H_range})
\begin{align*}
\sum_{q\in \bm{\mathcal{U'}}_{H}}\sum_{h\leq KH}\bm{\lambda}(H; q, n-qh)&=
\left(1+ O\left(\frac{1}{H^{1+\varepsilon}}\right)\right) \frac{\#\mathcal{Q}_{H}KH}{\sigma_2}=\\
&=\left(1+ O\left(\frac{1}{(\ln x)^{\delta(1+\varepsilon)}}\right)\right) \frac{\#\mathcal{Q}_{H}KH}{\sigma_2}.
\end{align*}
Therefore
\begin{align*}
\sum_{q\in \bm{\mathcal{U'}}}\sum_{h\leq KH_{q}}\bm{\lambda}(H; q, n-qh)&=
\sum_{H\in \mathfrak{H}_{1}}\sum_{q\in \bm{\mathcal{U'}}_{H}}\sum_{h\leq KH}\bm{\lambda}(H; q, n-qh)=\\
&=\left(1+O\left(\frac{1}{(\ln x)^{\delta (1+\varepsilon)}}\right)\right)C_{2} (K+1)y,
\end{align*}where
\[
C_2= \frac{K}{(K+1)y}\sum_{H\in \mathfrak{H}_{1}}\frac{\# \mathcal{Q}_{H} H}{\sigma_{2}}.
\]By \eqref{sigma_assympt}, we have
\[
\sigma_{2}=\sigma_{H^M, z}\sim \frac{C_{1}/\ln z}{C_{1}/\ln H^{M}}\sim \frac{M\ln H}{\ln x}.
\] Applying \eqref{QH_assympt}, we obtain
\[
C_2 \sim \frac{K(\xi-1)}{(K+1)M \xi}\sum_{H\in \mathfrak{H}_{1}} \frac{1}{\ln H}.
\]We have
\[
\sum_{H\in \mathfrak{H}_1}\frac{1}{\ln H}= \frac{1}{2\ln \xi} \sum_{A\leq j \leq B}\frac{1}{j},
\]where
\[
A=\frac{\ln (2 (\ln x)^{\delta})}{2\ln (2\xi)},\qquad B=\frac{\ln ( (\ln x)^{1/2}/ (\xi (\ln\ln x)^{2}))}{2 \ln (2\xi)}.
\]Since
\[
\sum_{n\leq x}\frac{1}{n}=\ln x+ \gamma + O\left(\frac{1}{x}\right),
\] we see that
\[
\sum_{A\leq j \leq B}\frac{1}{j}= \ln \frac{B}{A} + O\left(\frac{1}{A}\right)\sim \ln \left(\frac{1}{2\delta}\right).
\]Hence,
\begin{equation}\label{L2:C2_assympt}
C_2\sim \frac{K(\xi-1)}{2(K+1)M \xi\ln \xi} \ln \left(\frac{1}{2\delta}\right).
\end{equation}

Let $H\in \mathfrak{H}_{2}$. We put
\begin{equation}\label{L2:def_RH}
\bm{\mathcal{R}}_{H}=\left\{n\in \mathbf{S^{*}}\cap [-y,-1]: \left|\sum_{q\in \mathcal{Q}_{H}}\sum_{h\leq KH}\bm{\lambda^{*}}(H;q, n+qh) - \frac{\#\mathcal{Q}_{H} KH}{\sigma_{2}}\right|\geq \frac{\#\mathcal{Q}_{H} KH}{\sigma_{2} H^{(M-3)/2 - \varepsilon}}\right\}.
\end{equation}Similarly, we show that
\[
\mathbb{P}\Big(\# \bm{\mathcal{R}}_{H}\leq \frac{\sigma y}{H^{1+\varepsilon}}\Big)= 1-O\left(\frac{1}{H^{\varepsilon}}\right).
\]We put
\begin{equation}\label{L2:def_RH_str}
\bm{\mathcal{R'}}_{H}=\bigg\{n\in \mathbf{S^{*}}\cap [-y, -1]:
\sum_{q\in \mathcal{Q}_{H}\setminus \bm{\mathcal{V'}}_{H}}
 \sum_{h\leq KH}\bm{\lambda^{*}}(H; q, n+qh)\geq \frac{ \#\mathcal{Q}_{H}KH}{H^{1+\varepsilon}\sigma_2}\bigg\}.
\end{equation}Similarly, we show that
\[
\mathbb{P}\bigg(\#\bm{\mathcal{R'}}_{H} \leq \frac{\sigma y}{H^{1+\varepsilon}}\bigg)=1-O\left(\frac{1}{H^{M-6-5\varepsilon}}\right).
\] Hence, with probability $1 - O\big((\ln x)^{-\delta \eta}\big)$, where $\eta = \min (\varepsilon, M-6-5\varepsilon) >0$, the relations
\[
\#\bm{\mathcal{R}}_{H} \leq \frac{\sigma y}{H^{1+\varepsilon}},\qquad
\#\bm{\mathcal{R'}}_{H} \leq \frac{\sigma y}{H^{1+\varepsilon}}
\]hold for every $H\in \mathfrak{H}_{2}$ simultaneously. We put
\[
\bm{\mathcal{N^{*}}} = (\mathbf{S^{*}}\cap [-y,-1]) \setminus \bigcup_{H\in \mathfrak{H}_{2}}
(\bm{\mathcal{R}}_{H}\cup \bm{\mathcal{R'}}_{H}).
\] We obtain
\[
\# \bigcup_{H\in \mathfrak{H}_{2}}
(\bm{\mathcal{R}}_{H}\cup \bm{\mathcal{R'}}_{H})\ll \frac{\sigma y}{(\ln x)^{\delta (1+\varepsilon)}}
\]which is smaller than $x/ (10 \ln x)$ for large $x$. It remains to verify \eqref{L2:N_star_AP} for $n\in \bm{\mathcal{N^{*}}}$. Since $n\notin \bm{\mathcal{R}}_{H}$ and $n\notin \bm{\mathcal{R'}}_{H}$ for every $H\in \mathfrak{H}_{2}$, the inequalities opposite to those in \eqref{L2:def_RH} and \eqref{L2:def_RH_str} hold, and we have for each $H\in \mathfrak{H}_2$ the asymptotic
\begin{align*}
\sum_{q\in \bm{\mathcal{V'}}_{H}}\sum_{h\leq KH}\bm{\lambda^{*}}(H; q, n+qh)&=
\left(1+ O\left(\frac{1}{H^{1+\varepsilon}}\right)\right) \frac{\#\mathcal{Q}_{H}KH}{\sigma_2}=\\
&=\left(1+ O\left(\frac{1}{(\ln x)^{\delta(1+\varepsilon)}}\right)\right) \frac{\#\mathcal{Q}_{H}KH}{\sigma_2}.
\end{align*}
Therefore
\begin{align*}
\sum_{q\in \bm{\mathcal{V'}}}\sum_{h\leq KH_{q}}\bm{\lambda^{*}}(H; q, n+qh)&=
\sum_{H\in \mathfrak{H}_{2}}\sum_{q\in \bm{\mathcal{V'}}_{H}}\sum_{h\leq KH}\bm{\lambda^{*}}(H; q, n+qh)=\\
&=\left(1+O\left(\frac{1}{(\ln x)^{\delta (1+\varepsilon)}}\right)\right)C_{3} (K+1)y,
\end{align*}where
\[
C_3= \frac{K}{(K+1)y}\sum_{H\in \mathfrak{H}_{2}}\frac{\# \mathcal{Q}_{H} H}{\sigma_{2}}.
\]Similarly, we show that
\begin{equation}\label{L2:C3_assympt}
C_3\sim \frac{K(\xi-1)}{2(K+1)M \xi\ln \xi} \ln \left(\frac{1}{2\delta}\right).
\end{equation}We see from \eqref{Basic_delta}, \eqref{Basic_Delta_2}, \eqref{L2:C2_assympt} and \eqref{L2:C3_assympt} that if $M>6$ is sufficiently close to $6$, $K$ is large enough and $\xi>1$ is sufficiently close to $1$, then for sufficiently large $x$
\[
10^{2\delta} \leq C_2 \leq 100,\qquad 10^{2\delta} \leq C_3 \leq 100.
\]Lemma \ref{L2} is proved.

\end{proof}

\begin{lemma}\label{L3}
Suppose that $0<\delta \leq 1/2$, let $y\geq y_{0}(\delta)$ with $y_{0}(\delta)$ sufficiently large, and let $V$ be a finite set with $\#V\leq y$. Let $1\leq s \leq y$, and suppose that $\mathbf{e}_{1},\ldots, \mathbf{e}_{s}$ are random subsets of $V$ satisfying the following:
\begin{align}
\#\mathbf{e}_{i}&\leq \frac{(\ln y)^{1/2}}{\ln\ln y}\qquad (1\leq i\leq s),\label{L3:I}\\
\mathbb{P}(v\in \mathbf{e}_{i})&\leq y^{-1/2 - 1/100}\qquad (v\in V,\ 1\leq i \leq s),\label{L3:II}\\
\sum_{i=1}^{s}\mathbb{P}(v, v' \in \mathbf{e}_{i})&\leq y^{-1/2}\qquad (v, v' \in V,\ v\neq v'),\label{L3:III}\\
\bigg|\sum_{i=1}^{s} \mathbb{P}(v\in \mathbf{e}_{i}) - C_4 \bigg|&\leq \eta\qquad (v\in V),\label{L3:IV}
\end{align}where $C_4$ and $\eta$ satisfy
\[
10^{2\delta}\leq C_4\leq 100,\qquad \eta \geq \frac{1}{(\ln y)^{\delta}\ln\ln y}.
\]Then there are subsets $e_i$ of $V$, $1\leq i \leq s$, with $e_i$ being in the support of $\mathbf{e}_{i}$ for every $i$, and such that
\[
\# \bigg(V\setminus \bigcup_{i=1}^{s} e_i\bigg)\leq C_{5}\eta \#V,
\]where $C_5 >0$ is an absolute constant.
\end{lemma}
\begin{proof}
This is \cite[Lemma 3.1]{FKMPT}.
\end{proof}

We are now in a position to prove Theorem \ref{T1}. Let $b$, $\mathcal{U'}$, $\mathcal{V'}$, $\mathcal{N}$, and $\mathcal{N^{*}}$ be the quantities whose existence is asserted by Lemma \ref{L2}.

For each $q\in \mathcal{U'}$, we choose a random integer $\mathbf{n}_{q}$ with probability density function
\begin{equation}\label{Def_n}
\mathbb{P}(\mathbf{n}_{q}=n)=\frac{\lambda(H_{q}; q, n)}{\sum_{-Ky<n'\leq y}\lambda(H_{q}; q, n')}\qquad
(-Ky< n \leq y).
\end{equation} Note that by \eqref{L2:lambda} the denominator is non-zero, so that this is a well-defined probability distribution. For each $q\in \mathcal{U'}$, we define the random subset $\mathbf{e}_{q}$ of $\mathcal{N}$ by the formula
\begin{equation}\label{Def_e}
\mathbf{e}_{q}= \mathcal{N}\cap \{\mathbf{n}_{q}+hq:\ 1\leq h \leq KH_{q}\}.
\end{equation}

We are going to apply Lemma \ref{L3} with $V=\mathcal{N}$, $s=\# \mathcal{U'}$, $\{\mathbf{e}_{1},\ldots, \mathbf{e}_{s}\}=\{\mathbf{e}_{q}:\ q\in \mathcal{U'}\}$, $C_4=C_{2}$ given by Lemma \ref{L2}, and
\begin{equation}\label{eta}
\eta=\frac{1/40}{C_{5} (\ln x)^{\delta}}.
\end{equation}

If $q\in \mathcal{U'}$, then from \eqref{H_range} we have
\[
\#\mathbf{e}_{q}\leq K H_{q}\leq K \frac{(\ln x)^{1/2}}{(\ln \ln x)^{2}}\leq \frac{(\ln y)^{1/2}}{\ln \ln y},
\]if $x$ is large enough. Hence, \eqref{L3:I} holds.

For $n\in \mathcal{N}$ and $q\in \mathcal{U'}$, we have from \eqref{Def_n}, \eqref{Def_e}, \eqref{L2:lambda}, and \eqref{Def_lambda} that
\begin{align*}
\mathbb{P}(n\in \mathbf{e}_{q})&=\sum_{1\leq h \leq KH_{q}} \mathbb{P}(\mathbf{n}_{q}=n-hq)=
\sum_{1\leq h \leq KH_{q}} \frac{\lambda(H_{q}; q, n-hq)}{\sum_{-Ky<n'\leq y}\lambda(H_{q}; q, n')}\ll\\
&\ll \frac{1}{y}\sum_{1\leq h \leq KH_{q}}\lambda(H_{q}; q, n-hq)\ll \frac{1}{y} H_{q} \sigma_{2}^{-KH_q}.
\end{align*}Let us show that
\begin{equation}\label{Ineq_Hq}
H_{q} \sigma_{2}^{-KH_q}\leq x^{1/10}
\end{equation}or, that is equivalent,
\[
\ln H_{q} + K H_{q}\ln (\sigma_{2}^{-1})\leq \frac{\ln x}{10}.
\]We have (if $x$ is large enough)
\[
H_{q}\leq \frac{(\ln x)^{1/2}}{(\ln\ln x)^{2}}\leq \ln x
\] and hence $\ln H_{q} \leq \ln\ln x$. Also,
\[
\sigma_{2}^{-1}\sim \frac{\ln z}{M\ln H}\sim \frac{\ln x}{M\ln H}.
\]Therefore
\[
\sigma_{2}^{-1}\leq \frac{2\ln x}{M\ln H}\leq  \frac{2\ln x}{M\delta\ln \ln x}\leq \ln x,
\]if $x$ is large enough. We obtain
\[
\ln H_{q} + K H_{q}\ln (\sigma_{2}^{-1})\leq \ln\ln x +
K \frac{(\ln x)^{1/2}}{\ln\ln x}\leq \frac{\ln x}{10},
\]if $x$ is large enough. Thus, \eqref{Ineq_Hq} is proved. We obtain
\[
\mathbb{P}(n\in \mathbf{e}_{q})\ll \frac{1}{y^{9/10}}
\]which gives \eqref{L3:II} for $x$ large enough.

For $n\in \mathcal{N}$, we have from \eqref{Def_n}, \eqref{Def_e}, \eqref{L2:lambda}, and \eqref{L2:lambda_AP} that
\begin{align*}
\sum_{q\in \mathcal{U'}} \mathbb{P}(n\in \mathbf{e}_{q})&=
\sum_{q\in \mathcal{U'}}\sum_{1\leq h \leq KH_{q}} \mathbb{P}(\mathbf{n}_{q}=n-hq)=\\
&=
\sum_{q\in \mathcal{U'}}\sum_{1\leq h \leq KH_{q}}\frac{\lambda(H_{q}; q, n-hq)}{\sum_{-Ky<n'\leq y}\lambda(H_{q}; q, n')}=\\
&= C_{2} + O\big((\ln x)^{-\delta - \varepsilon}\big),
\end{align*}and \eqref{L3:IV} follows.

We now turn to \eqref{L3:III}. Observe from \eqref{Def_e} that for distinct $n, n' \in \mathcal{N}$, one can only have $n, n' \in \mathbf{e}_{q}$ if $q$ divides $n-n'$. Since $|n-n'|\leq 2 y$ and $q\geq z > \sqrt{2y}$, there is at most one $q\in \mathcal{U'}$ for which this is the case, and \eqref{L3:III} now follows from \eqref{L3:II}.

Thus, all assumptions of Lemma \ref{L3} hold. By Lemma \ref{L3}, for each $q\in \mathcal{U'}$ there is a number $n_{q}$ such that if we put
\[
e_q := \mathcal{N}\cap \{n_{q}+hq:\ 1\leq h \leq KH_{q}\} \qquad (q\in \mathcal{U'}),
\] then we have
\begin{equation}\label{N_minus}
\# \Big(\mathcal{N}\setminus \bigcup_{q\in \mathcal{U'}} e_{q}\Big)\leq C_{5}\eta\#\mathcal{N}.
\end{equation}Since $C_1 = e^{-\gamma}< 1$ in \eqref{sigma_assympt}, we have
\[
\sigma\leq \frac{1}{\ln x},
\]if $x$ is large enough. By \eqref{L2:i_S}, we have
\begin{equation}\label{N_EST}
\#\mathcal{N} \leq \# (S\cap [1, y])\leq 2 \sigma y\leq \frac{2(x (\ln x)^{\delta} + 1)}{\ln x}\leq
\frac{4 x (\ln x)^{\delta}}{\ln x}.
\end{equation}From \eqref{N_minus},  \eqref{eta} and \eqref{N_EST} we obtain
\[
\# \Big(\mathcal{N}\setminus \bigcup_{q\in \mathcal{U'}} e_{q}\Big)\leq \frac{x}{10 \ln x}.
\]

For each $q\in \mathcal{V'}$, we choose a random integer $\mathbf{m}_{q}$ with probability density function
\[
\mathbb{P}(\mathbf{m}_{q}=n)=\frac{{\lambda}^{*}(H_{q}; q, n)}{\sum_{-y \leq n'< Ky}{\lambda}^{*}(H_{q}; q, n')}\qquad
(-y\leq n < Ky).
\] Note that by \eqref{L2:lambda_star} the denominator is non-zero, so that this is a well-defined probability distribution. For each $q\in \mathcal{V'}$, we define the random subset $\mathbf{r}_{q}$ of $\mathcal{N^{*}}$ by the formula
\[
\mathbf{r}_{q}= \mathcal{N^{*}}\cap \{\mathbf{m}_{q}-hq:\ 1\leq h \leq KH_{q}\}.
\]

We are going to apply Lemma \ref{L3} with $V=\mathcal{N^{*}}$, $s=\# \mathcal{V'}$, $\{\mathbf{e}_{1},\ldots, \mathbf{e}_{s}\}=\{\mathbf{r}_{q}:\ q\in \mathcal{V'}\}$, $C_4=C_{3}$ given by Lemma \ref{L2}, and $\eta$ given by \eqref{eta}.

 Similarly, we show that all assumptions of Lemma \ref{L3} hold. By Lemma \ref{L3}, for each $q\in \mathcal{V'}$ there is a number $m_{q}$ such that if we put
\[
r_q := \mathcal{N^{*}}\cap \{m_{q}-hq:\ 1\leq h \leq KH_{q}\} \qquad (q\in \mathcal{V'}),
\] then we have
\[
\# \Big(\mathcal{N^{*}}\setminus \bigcup_{q\in \mathcal{V'}} r_{q}\Big)\leq \frac{x}{10 \ln x}.
\]

We take $b\equiv n_{q}$ (mod $q$) for all $q\in \mathcal{U'}$ and $b\equiv -N - m_q$ (mod $q$) for all $q\in \mathcal{V'}$ (since $\mathcal{U'}\cap \mathcal{V'}=\emptyset$, by the Chinese Remainder Theorem we can do this). Since
\[
e_{q}\subset \{n\in \mathcal{N}:\ n\equiv n_{q}\ \text{(mod $q$)}\},
\]we find that
\begin{align*}
\# (S_{x/2}(b)\cap [1,y])&\leq \# \big((S\cap [1, y])\setminus \mathcal{N}\big)+
\# \Big(\mathcal{N}\setminus \bigcup_{q\in \mathcal{U'}}e_{q}\Big)\leq\\
&\leq \frac{x}{10 \ln x} + \frac{x}{10 \ln x} = \frac{x}{5 \ln x}.
\end{align*}Similarly, since
\[
r_{q}\subset \{n\in \mathcal{N^{*}}:\ n\equiv m_{q}\ \text{(mod $q$)}\},
\]we find that
\[
\# (S_{x/2}(-N -b)\cap [-y,-1])\leq  \frac{x}{5 \ln x}.
\]

Let us denote $\mathcal{A}:= S_{x/2}(b)\cap [1,y]$, $\mathcal{D}:= S_{x/2}(-N -b)\cap [-y,-1]$, $L_1 :=\{p: x/2< p\leq (3x)/4\}$,
$L_2:=\{p: (3x)/4 < p \leq x$\}. Then we have
\[
\# L_1 > \frac{x}{5\ln x}\geq \#A,\qquad \# L_2 > \frac{x}{5\ln x}\geq \#D,
\]if $x$ is large enough. Hence, we may pair up each element $a\in \mathcal{A}$ with a unique prime $q=q_{a}\in L_{1}$, and pair up each element $d\in \mathcal{D}$ with a unique prime $q=q_{d}\in L_{2}$. We take $b\equiv a$ (mod $q_{a}$) for every $a\in \mathcal{A}$ and $b\equiv -N - d$ (mod $q_{d}$) for every $d\in \mathcal{D}$ (since $L_1 \cap L_{2}=\emptyset$, by the Chinese Remainder Theorem we can do this). We obtain
\[
S_{x}(b)\cap [1,y]=\emptyset,\qquad S_{x}(-N -b)\cap [-y,-1]=\emptyset.
\]

Let $x=(1/4)\ln N$ and let $N$ be large enough. Then
\[
N^{1/5} \leq P_{x}\leq N^{1/3},\qquad \frac{1}{5} \ln N (\ln \ln N)^{\delta} \leq y \leq \frac{1}{3}\ln N (\ln \ln N)^{\delta},
\] if $N$ is large enough. Let $b_1\in \mathbb{Z}$ be such that $b_1\equiv b$ (mod $P_{x}$) and $-4 P_{x}\leq b_{1}< - 3P_{x}$. Then $|b_1|<N/9$ (if $N$ is large enough) and
\[
S_{x}(b_1)\cap [1,y]=\emptyset,\qquad S_{x}(-N -b_{1})\cap [-y,-1]=\emptyset.
\]
Let $b_{2}:= -b_1$. Then
\[
S_{x}\cap (b_2 + [1, y]) = \emptyset, \qquad S_{x}\cap (N-b_2 + [-y,-1])=\emptyset.
\]Let $I_{1}= b_2 + [1, y]$, $I_{2}= N-b_2 + [-y,-1]$. We have
\[
x< b_2+1 < b_2+y < \frac{N}{4},\qquad \frac{3N}{4}< N-b_2-y< N-b_2-1<N,
\]if $N$ is large enough. In particular, all numbers $n\in I_1 \cup I_{2}$ are composite. We put $n_1 = b_2 + [y/2]$, $n_2= N-b_2 - [y/2]$. Hence, $n_1$ and $n_2$ are positive integers, $n_1+ n_2 = N$ and
\[
f(n_i)\geq [y/2]\geq \frac{1}{11} \ln N (\ln \ln N)^{\delta},\qquad i=1, 2.
\]Theorem \ref{T1} is proved.

\section{Acknowledgements}

I thank Sergei Konyagin for introducing the author to this problem.

\end{document}